\documentclass[a4paper,11pt]{amsart}
\usepackage[british]{babel}
\usepackage[colorlinks=true,citecolor=black,urlcolor=black,linkcolor=black]{hyperref}
\usepackage{booktabs,cite,amssymb,xcolor,microtype,enumitem}

\newtheorem{theorem}{Theorem}
\newtheorem*{theorem*}{Theorem}
\newtheorem{lemma}{Lemma}[section]
\theoremstyle{definition}
\newtheorem{definition}[lemma]{Definition}
\theoremstyle{remark}
\newtheorem{remark}[lemma]{Remark}

\newcommand{\abs}[1]{\left|#1\right|} 
\newcommand{\cC}{\mathcal{C}} 
\newcommand{\bC}{\mathbb{C}} 
\newcommand{\cH}{\mathcal{H}} 
\newcommand{\cL}{\mathcal{L}} 
\newcommand{\cM}{\mathcal{M}} 
\newcommand{\bN}{\mathbb{N}} 
\newcommand{\norm}[1]{\left\|#1\right\|} 
\newcommand{\cP}{\mathcal{P}} 
\newcommand{\cQ}{\mathcal{Q}} 
\newcommand{\bR}{\mathbb{R}} 
\newcommand{\bS}{\mathbb{S}} 
\newcommand{\bZ}{\mathbb{Z}} 
\newcommand{\Leb}{\operatorname{Leb}} 
\newcommand{\dist}[2]{\operatorname{dist}(#1,#2)} 
\newcommand{\ddist}[3]{\operatorname{dist}_{#1}(#2,#3)} 

\begin{document}
\title{Exponential mixing for singular skew-products}

\author{Oliver Butterley}
\address{(Oliver Butterley) Department of Mathematics, University of Rome Tor Vergata -- Via della Ricerca Scientifica 1 -- 00133 Roma -- Italy}
\email{butterley@mat.uniroma2.it}

\begin{abstract}
  We study skew-products of the form \((x,u) \mapsto (fx, u + \varphi(x))\) where \(f\) is a non-uniformly expanding map on a manifold \(X\) and \(\varphi: X \to \bS^1\) is piecewise \(\cC^1\).
  If the systems satisfies mild assumptions (in particular singular behaviour of \(\varphi\) is permitted) then we prove that the map mixes exponentially with respect to the unique SRB measure.
  This extends previous results by allowing singular behaviour in the fibre map.
\end{abstract}

\maketitle
\thispagestyle{empty}

\section{Introduction}
\label{sec:intro}

A deceptively simple transformation on \([0,1)^2\) is defined as the skew-product \((x,u) \mapsto ([2x], [u + \dist{x}{b}^{a}])\) for some \(a\in(0,1)\), \(b\in [0,1)\).
Although simple to write this example presents an interesting and non-trivial difficulty.
This is the motivating example for this work.
The present work is devoted to exploring the technology which can be used to prove exponential mixing for this and other settings.
We will permit a rather general setting (arbitrary dimension, general classes of maps).
Relatively few concrete examples which have a neutral direction are known to mix exponentially, here we increase this collection.

Skew products with some resemblance to this example have been studied by observing that there is an invariant unstable cone field (in \cite{BE17} this was explicit, in \cite{BV05,AGY06} this idea was still present in the assumptions).
However this isn't satisfied by the above mentioned system.
The singular behaviour which we permit in this present work is a significant problem for the established methods because it is impossible to have an invariant unstable cone field which is uniformly bounded away from the neutral direction. For want of a better expression we call this ``unbounded twist''.
Nevertheless we show that the singular behaviour does not prevent good limit theorems and, in particular, our results prove that the motivating example mixes exponentially.

In this present work we solve the unbounded twist issue by inducing.
We present a general framework and then develop the application to a specific example.
In cases like the above example we will induce even though it seems that the base map doesn't need any inducing.
This ``over-inducing'' suffices to solve the unbounded twist issue.
The key idea is to, if required, view the singularities of the fibre map as artificial singularities of the map in the construction of the induced system. All of this means that we demonstrate that a very large class of partially hyperbolic systems mix exponentially.

Let \(f\) be a transformation on a compact manifold \(X\) and let \(\varphi: X \to \bS^1\) be a function from manifold to the circle \(\bS^1\). We will call \(f\) the \emph{base map} and \(\varphi\) the \emph{fibre map}.
This article is devoted to the study of partially hyperbolic systems defined as the skew-product \(f_\varphi: X \times \bS^1 \to  X \times \bS^1 \),
\[
  f_\varphi : (x,u) \mapsto (fx, u + \varphi(x)).
\]
In order to prove the results we will use the existence of an induced system (Young tower with exponential tails) in the sense that there exists a connected subset \(Y \subset X\) and a piecewise constant inducing time \(R:Y\to\bN\) such that the induced map \(F : x \to f^{R(x)}(x)\) is a full branch Markov map, \(\cC^2\) on partition elements. Defining the induced fibre map as \(\Phi(x) = \sum_{k=0}^{R(x)-1}\varphi(f^k x)\) we consider the induced skew-product \(F_\Phi : Y \times \bS^1 \to  Y \times \bS^1 \),
\[
  F_\Phi : (x,u) \mapsto (Fx, u + \Phi(x)).
\]
The principal aim of this work is to allow weak control on the fibre map, in particular to allow \(D(\varphi \circ f|_\omega^{-1})\) to be unbounded. In one part of this work we will show that if the induced skew-product satisfies certain assumptions then the original system mixes exponentially. In the other part we introduce assumptions (with the emphasis on verifiability) which suffice to show that the induced system satisfies the previously mentioned assumptions.

The problem of exponential mixing for partially hyperbolic maps like these skew-products (similar to flows) is rather difficult because of the neutral direction and the singularities which we permit in the fibre map. Because of the neutral direction, in order to study the rate of mixing and other strong statistical properties, we must use some form of Dolgopyat estimate and observe oscillatory cancellations~\cite{Dolgopyat98}. Recent years have seen significant progress for such results for systems with a neutral direction (e.g., \cite{Pollicott85,Tsujii08,ABV14,BE17,Tsujii18,BW20,CL22}), all using methods based on the work of Dolgopyat to some extent. In particular the results have been extended to the Lorenz flow and systems inspired by it~\cite{BM18,Butterley14,AMV15,AM16,AM17}. Our purpose here is to give a general result for skew-products with (possibly) some degree of singular behaviour and in the process clarify exactly the reach and limits of the notions.
Similar systems to the ones studied here were previously introduced~\cite{MPV13} as a model for the Lorenz flow (in this reference it was claimed that these systems mix exponentially but there was a gap in the proof).

Although the neutral direction causes significant technical difficulty, in our case we fortunately can use (to a large extent) the work of Gouëzel~\cite{Gouezel09} which includes the study of induced skew-products, in particular used to obtain results concerning Farey sequences.
As observed before, a key point here is that we induce even if the base map is already uniformly hyperbolic, in order to obtain the required property for the induced fibre map. In some applications, including the motivating example, a large deviations result for the expanding map (for a function based on distance to singularity) implies that the (artificially) induced system has trajectories which behave well with respect to the singularities of the fibre map. Using this together with the fact that the singularities of the fibre map aren’t worse than distance to some negative power, means that the induced fibre map has the required regularity.

\begin{remark}
  Most likely these results would extend easily to the hyperbolic case when the stable foliation is at least \(\cC^1\) by disintegration along the stable foliation (e.g., using \cite{BM17} in a similar way as was done in \cite{BBM19}).
  However such regularity is arguably not typical in relevant settings~\cite{HW99}.
\end{remark}

\begin{remark}
  We will mostly avoid the details related to possible discontinuities in the system.
  Discontinuities require delicate control, particularly in higher dimension when oscillatory cancellation arguments are required (e.g., \cite{Saussol00,Eslami17,Eslami20,EMV21}).
  However, inducing, as we do here for other motives, can be useful for avoiding the problems related to discontinuities (e.g., \cite{Obayashi09}).
\end{remark}

\begin{remark}
  Suspension semiflows and skew products as we study here share many similarities, in particular the investigation of the rate of mixing requires exactly the same estimates (the ``twisted transfer operators'' are identical). 
  As such, the results here can be relatively easily transferred to the corresponding suspension semiflow setting.
  Alternatively, the extension to general compact group extensions (see e.g., \cite{Dolgopyat02}) is feasible although non-trivial and in this work we choose not to take this road.
\end{remark}

\begin{remark}
  In this present work we assume that the base map is at least \(\cC^2\). Most likely the results can be extended to the \(\cC^{1}\) plus Hölder derivative case by already established ideas (e.g., \cite{AM16, BW20}).
\end{remark}

\begin{remark}
  In this present work our focus is on exponential mixing but there are various other relevant statistical properties, for example, central limit theorem, local limit theorem, almost-sure invariance principal, etc.
  It is to be expected that the main estimates we obtain during this work and which are used for proving the exponential mixing can also be used for proving the other statistical limit laws (see e.g., arguments contained in \cite{AM16,AMV15}).
  (See also \cite{GP10,GRS15,Galatolo18} for some related results.)
\end{remark}

\section{Setting \& results}
\label{sec:results}

Although not part of the motivating example, technology developed for non-uniformly expanding (NUE) systems will be invaluable for tackling the problem. The following four conditions are as introduced by Young~\cite{Young98,Young99} and this structure is often called a \emph{Young tower} structure.

\begin{definition}%
  \label{def:NUE1}
  We say that \(f:X\to X\) is \emph{NUE in the sense of Young} to mean the tuple \((f,X,\mu,Y,R)\) where:
  \begin{itemize}[leftmargin=1.4em]
    \item   \(X\) is a compact Riemannian manifold (possibly with boundary), endowed with a Borel measure \(\mu\) (called the reference measure);
    \item   \(f\) is a nonsingular transformation on \(X\);
    \item   \(Y\) (called the base of the induced system) is a connected open subset of \(X\) with finite measure and finite diameter and there exists a finite or countable partition \({\{Y_\ell\}}_{\ell \in \Lambda}\) of a full measure subset of \(Y\);
    \item   \(R:Y\to\bN\)  (called the return time) is a function constant on each partition element \(Y_{\ell}\);
  \end{itemize}
  Such that the following properties are satisfied:
  \begin{description}[itemindent=-0.3em,leftmargin=1em]
    \item[Y1]
      For each \(\ell \in \Lambda\) let \(R_\ell\) denote the constant value that \(R\) takes on \(Y_{\ell}\).
      For all \(\ell \in \Lambda\), the restriction of \(f^{R_\ell}\) to \(Y_\ell\) is a diffeomorphism between \(Y_\ell\) and \(Y\), satisfying \(\kappa \norm{v} \leq \norm{\smash{Df^{R_\ell}(x)v}} \leq C_\ell \norm{v}\) for any \(x\in Y_\ell\) and for any tangent vector \(v\) at \(x\), for some constants \(\kappa > 1\) (independent of \(\ell\)) and \(C_\ell\). We denote by \(F: Y \to Y\) the map which is equal to \(f^{R_\ell}\) on each set \(Y_\ell\).
    \item[Y2]
      Let \(\cH^{n}_{F}\) denote the set of inverse branches of \(F^n\). Let \(J(x)\) be the inverse of the Jacobian of \(F\) at \(x\) with respect to \(\mu\). We assume that there exists a constant \(C>0\) such that, for any \(\xi\in\cH^{1}_{F}\), \(\norm{D((\log J)\circ \xi)}\leq C\).
    \item[Y3]
      There exists a constant \(C>0\) such that, for any \(\ell\), if \(\xi_\ell: Y \to Y_\ell\) denotes the corresponding inverse branch of \(F\), for any \(k\leq R_\ell\), then \(\norm{\smash{f^k \circ \xi_\ell}}_{\cC^1(Y)}\leq C\).
  \end{description}
\end{definition}

\begin{definition}
  Suppose that  \(f:X\to X\) is NUE in the sense of Young. Then \(f\) is said to have \emph{exponential tails} if there exists \(\sigma_0>0\) such that \(\int_Y e^{\sigma_0 R} \ d\mu < \infty\).
\end{definition}

If \(f:X\to X\) is NUE in the sense of Young with exponential tails then it is known~\cite{Young98} that there exists a probability measure \(\tilde \mu \) on \(X\) which is absolutely continuous with respect to \(\mu\), invariant under \(f\) and ergodic. Moreover, if \(f\) is mixing for \(\tilde \mu\), then it is exponentially mixing (for H\"older continuous observables).

\begin{definition}
  An open subset of a Riemannian manifold \(U\)  is said to have the \emph{weak Federer property} with respect to a finite Borel measure \(\nu\), if, for any \(\gamma>1\), there exists \(D = D(U,\gamma)>1\) and \(\eta_0(\gamma)>0\) such that, for any \(\eta \in (0,\eta_0(\gamma))\),
  \begin{itemize}[leftmargin=1.4em]
    \item There exists a set of points \({\{x_j\}}_{j=1}^{k}\) such that the balls \(B(x_j,\gamma\eta)\) are disjoint and compactly included in \(U\);
    \item  There exists a set of sets \({\{A_j\}}_{j=1}^{k}\) whose union covers a full measure subset of \(U\) and \(A_j \subset B(x_j,\gamma  \eta D)\);
    \item For any \(y_j \in B(x_j, (\gamma-1)\eta)\), we have \(\nu(B(y_j,\eta))\geq D^{-1}\nu(A_j)\).
  \end{itemize}
\end{definition}

\begin{definition}
  A family of open sets \({\{U_n\}}_{n}\) is said to uniformly have the weak Federer property for the measure \(\nu\) if, for all \(\gamma>1\), \(\sup_{n}D(U_{n},\gamma)\) is finite.
\end{definition}

\begin{definition}
  Suppose that \(f : X \to X\) is NUE in the sense of Young. We say that the transformation has the \emph{weak Federer property} if, for each \(h \in \bigcup_{n\in\bN} \cH^{n}_{F}\), the sets \(h(Y)\) uniformly have the weak Federer property with respect to \(\mu_Y\) (the probability measure induced by \(\mu\) on \(Y\)).
\end{definition}

If \({\{U_n\}}_{n}\) is a family of open intervals then the uniform Federer property is trivially satisfied by Lebesgue measure (see \cite[\S6.1]{Gouezel09} for a general criterion for the weak Federer property and see \cite[Remark 2.1]{BV05} for additional comments).
In order to prove exponential mixing uniformity in the Federer assumption is not required~\cite[Remark 2.5]{Gouezel09}.

\begin{definition}
  \label{def:cohomol}
  Let \(Y\) be a set as above with partition \({\{Y_\ell\}}_{\ell\in \Lambda}\).
  A function \(\Phi : Y \to \bS^1\) is said to be \emph{cohomologous to a locally constant function} if there exists a \(\cC^1\) function \(\psi:Y \to \bS^1\) such that \(\Phi - \psi + \psi \circ F\) is constant on each set \(Y_\ell\), \(\ell\in \Lambda\).
\end{definition}

The skew product transformation \(f_\varphi : (x,u) \mapsto (fx, u + \varphi(x))\) is an isometry in the fibres and hence preserves the measure \(\nu = \tilde \mu \times m\) (where \(m\) denotes Lebesgue measure on \(\bS^1\)).

\begin{theorem}
  \label{thm:exp-one}
  Suppose that \(f : X \to X\) is  NUE in the sense of Young with exponential tails and and satisfying the weak Federer property.
  Suppose that \(\varphi : X \to \bR\) is \(\cC^1\) on the interior of \(X\), that the induced function \(\Phi(x) = \sum_{k=0}^{R(x)-1}\varphi(f^k x)\) is not cohomologous to a locally constant function and \(\norm{\smash{D\Phi(x)DF(x)^{-1}}}\) is uniformly bounded for \(x\in Y\).

  Then \(f_\varphi\) mixes exponentially for observables on \(Y \times \bS^1\) in the sense that:
  For any \(\alpha>0\) there exists \(\theta \in (0,1)\), \(C>0\) such that, for all functions  \(g\), \(h\) from \(Y\times \bS^1\) to \(\bC\), bounded and H\"older continuous with exponent \(\alpha\), and for all \(n\in \bN\),
  \[
    \abs{
      \int g \circ f_\varphi^n \cdot h \ d\nu - \left(\int g  \ d\nu\right)\left(\int h  \ d\nu\right)
    }
    \leq C \theta^n \norm{g}_{L^\infty}\norm{h}_{\cC^\alpha}.
  \]
\end{theorem}

This theorem is essentially the work of Gouëzel~\cite[Theorem 1.7]{Gouezel09} although the work of the reference requires the fibre map \(\varphi\) to be \(\cC^1\) on \(X\) whereas we allow unbounded derivative. However we require the induced fibre map \(\Phi\) to satisfy the same conditions as are actually required during the proof in the reference. On the other hand the result stated here is only for observables supported on the base of the tower \(Y\). Details concerning the modification of the reference in order to prove the theorem are given in Section~\ref{sec:sing-skew}.

\begin{remark}
  The restriction that the observables are supported on \(Y\times \bS^1\) can, to some extent, be mitigated in a standard way by considering observables which map to observables supported on \(Y\times \bS^1\) in finite steps.
  As such, typically exponential mixing results can be extended to observables supported on the complement of the singular set (depending on the exact construction of the inducing scheme).
\end{remark}

The assumptions of the above theorem are overly abstract from our point of view and so we would like to obtain some more verifiable conditions.
(For the origin of the following assumptions see \cite{ABV00,ALP05,Gouezel06}.) We use the following notation: For  $\delta>0$, set $\ddist{\delta}{x}{S}=\dist{x}{S}$ if $\dist{x}{S}<\delta$, and $\ddist{\delta}{x}{S}=1$ otherwise.

\begin{definition}%
  \label{def:NUE2}
  Let \(f\) be a map on a compact Riemannian manifold \(X\) (possibly with boundary).
  We assume that there exists a closed subset $S\subset M$, with zero Lebesgue measure (containing possibly discontinuities or critical points of \(f\) and with $\partial X \subset S$), such that $f$ is a $C^2$ local diffeomorphism on $X\setminus S$.
  We say that \(f\) is \emph{NUE in the sense of a controlled singular set} if the following assumptions are satisfied:
  \begin{enumerate}[label=\textbf{(S\arabic*):},ref=(S\arabic*)]
    \item \label{it:S1}
          (non-degeneracy close to $S$)
          We assume that there exist $B>1$ and $\beta>0$ such that,
          for any $x\in
            M\setminus S$ and every
          $v\in T_x M \setminus \{0\}$,
          \begin{equation*}
            \frac{1}{B}{\dist{x}{S}}^\beta \leq \frac{\norm{Df(x)v}}{\norm{v}}
            \leq B \dist{x}{S}^{-\beta}.
          \end{equation*}
          Assume also that, for all $x,y\in X$ with
          $\dist{x}{y}<\dist{x}{S}/2$,
          \begin{equation*}
            \Bigl|\log \norm{\smash{Df(x)^{-1}}} -\log \norm{\smash{Df(y)^{-1}}} \Bigr|
            \leq B \frac{\dist{x}{y}}{\dist{x}{S}^\beta}
          \end{equation*}
          and
          \begin{equation*}
            \bigl|\log |\det Df(x)^{-1}| - \log |\det Df(y)^{-1}| \bigr|
            \leq B \frac{\dist{x}{y}}{\dist{x}{S}^\beta}.
          \end{equation*}
    \item \label{it:S2}
          (points which are too close to $S$ or haven't yet experienced expansion)
          Let $\delta : (0,\epsilon_0) \to \bR_{+}$,  $\lambda>0$,
          \[
            \begin{aligned}
              \cP_{\epsilon,N} & = \left\{ x \in X :
              \frac{1}{n}\sum_{k=0}^{n-1}- \log
              \ddist{\delta(\epsilon)}{f^k x}{S} > \epsilon, \
              \text{for some \(n \geq N\)}\right\},  \\
              \cQ_{\epsilon,N} & = \left\{ x \in X :
              \frac{1}{n}\sum_{k=0}^{n-1} \log \norm{Df(f^k x)^{-1}}^{-1} < \lambda, \
              \text{for some \(n \geq N\)}\right\}.
            \end{aligned}
          \]
          We assume that there exists \(C>0\) and \(\theta\in(0,1)\) such that, for all $\epsilon \in (0,\epsilon_0)$, the Lebesgue measure of  \(\cP_{\epsilon,N} \cup \cQ_{\epsilon,N}\)
          is not greater than \(C\theta^N\).
  \end{enumerate}
\end{definition}

In order to take advantage of the above assumptions and build Young tower structures, the key concept of \emph{hyperbolic times} is used.
In this present work we further take advantage of this notion to deal with potential problems in the fibre map.
As such, let us now recall the definition of hyperbolic times.
For the purpose of this definition \(f:X\to X\) is a differentiable transformation and \(S\subset X\) is the \emph{singularity set}.
The following definition is exactly as used by Alves, Luzzatto \& Pinheiro~\cite{ALP05}, including the same notation.

\begin{definition}
  \label{def:hyp-time}
  Let \(b>0\), \(\sigma \in (0,1)\), \(\delta >0\).
  We say that \(n\in \bN\) is a \emph{\((b,\sigma,\delta)\)-hyperbolic time}\footnote{In the reference the terminology ``\((\sigma,\delta)\)-hyperbolic time'' is used and dependence on \(b\) is suppressed.} for \(x\) if, for all \(1\leq k \leq n\)
  \[
    \prod_{j=n-k}^{n-1} \norm{Df(f^jx)^{-1}} \leq \sigma^k
    \quad
    \text{and}
    \quad
    \ddist{\delta}{f^{n-k}x}{S} \geq \sigma^{bk}
  \]
  We will denote by \(H_n(b,\sigma,\delta)\) the set of points for which \(n\) is a \((b,\sigma,\delta)\)-hyperbolic time.
\end{definition}

The following is due to Gouëzel (result described in \cite[Proposition 1.16]{Gouezel09} with a proof which uses mostly \cite{Gouezel06}).

\begin{theorem*}[{\!\!\cite[Proposition 1.16]{Gouezel09}}]
  Suppose that \(X\) is a compact Riemannian manifold and that \(f:X\to X\) is NUE in the sense of a controlled singular set (Definition~\ref{def:NUE2}) and let \(b>0\) sufficiently small.
  There exists \(\sigma \in (0,1)\), \(\delta >0\)
  and there exists an open and connected subset \(Y\) of \(X\) such that \(f\) is NUE in the sense of Young (Definition~\ref{def:NUE1}) (on base \(Y\) with respect to Lebesgue measure) with exponential tails and and satisfying the weak Federer property.
  Moreover the return times for the Young tower are \((b,\sigma,\delta)\)-hyperbolic times.
\end{theorem*}

Note that the final statement isn't highlighted in the statement of the result in the cited reference although it is described in the argument \cite[\S2]{Gouezel06}.
For us this detail is important as already hinted, since we will use these hyperbolic times to control of the singular behaviour of the fibre maps.

\begin{theorem}
  \label{thm:exp-two}
  Let \(f:X\to X\) be \emph{NUE in the sense of a controlled singular set} \(S \subset X\) (Definition~\ref{def:NUE2}).
  Suppose that \(\varphi : X \to \bR\) is \(\cC^1\) on the connected components of \(X \setminus S\) and that there exists \(C>0\), \(s\geq 0\) such that, for all \(x\), \(\norm{D\varphi(x)Df(x)^{-1}} \leq C \dist{x}{S}^{-s}\).
  Further suppose that the induced function \(\Phi(x) = \sum_{k=0}^{R(x)-1}\varphi(f^k x)\) is not cohomologous to a locally constant function.
  Then there exists a subset \(Y\subset X\) such that the assumptions of Theorem~\ref{thm:exp-one} are satisfied.
\end{theorem}

\noindent
The above result is proven in Section~\ref{sec:twist}, using an argument based on hyperbolic times and which controls the regularity of the induced fibre map \(\Phi\).

In the case of our motivating example \((x,u) \mapsto ([2x], [u + {\dist{x}{1}}^{a}])\) we choose the singularity set \(S = \{1\}\) even thought \(1\) isn't a singular point in any sense for the base transformation \(x\mapsto 2x\).
In cases like this one needs to know that the NUE property of the system still holds, even when the singularity set is increase in order to consider also the singularities of the fibre map.
This can be done with minor restrictions and follows from a type of large deviations estimate.
The statement and proof of this in general settings is the content of Section~\ref{sec:artificial} (as done in~\cite{AV12}).

Returning to our motivating example, we have the following result.

\begin{theorem}
  \label{thm:example}
  Let \(X = [0,1]\) and \(f:X\to X\) is defined as \(f: x \mapsto 2x \mod 1\) and let \(\varphi : X \to \bS^1\) be defined as \(\varphi: x \mapsto {\dist{x}{1}}^{a}\) for some \(a\in (0,1)\).
  Then the skew-product
  \(f_\varphi : (x,u) \mapsto (fx, u + \varphi(x))\)
  mixes exponentially for observables supported on the complement of a neighbourhood of \(\{1\} \times \bS^1\).
\end{theorem}

Section~\ref{sec:uni} contains the proof of the above and discussion related to showing the property of the fibre map not being cohomologous to a local constant function in diverse settings.
Since we can for this specific example, we take a very hands-on approach to the argument and explicitly construct the tower and prove the required properties so that the above results can be applied.

\section{Singular skew-products}
\label{sec:sing-skew}

In this section we assume that the induced skew-product map satisfies the assumptions and show that this implies exponential mixing for the original skew-product.
This means that we prove Theorem~\ref{thm:exp-one}.
As made clear by Ruelle~\cite{Ruelle83} it is important to have some condition for the fibre map.
A common way to prove the results that we would like is to obtain a spectral gap for the transfer operator of the system (acting on a suitable Banach space) but the neutral direction complicated this problem (except for the work of Tsujii~\cite{Tsujii08,Tsujii18}).

Let \(f\) be NUE in the sense of Young with exponential tails.  preserving the probability measure \(\mu\).
Assume that \(\mu_Y\) has full support in \(Y\).
Let \(\varphi : X \to \bS^1\) be a \(\cC^1\) function such that the induced fibre map  \(\Phi(x)\) is not cohomologous to a locally constant function.
Let \(\nu = \mu \otimes \Leb \). Since \(f_\varphi\) is an isometry in the fibre the measure \(\nu\) is \(f_\varphi\)-invariant.
We consider the skew-product \(f_\varphi : (x,u) \mapsto (fx, u + \varphi(x))\).

This result is essentially what Gouëzel proved~\cite[\S3]{Gouezel09} however there is something that needs to be observed about their assumption on the fibre map.
Their results are stated for the case when the fibre map is \(\cC^1\) but such a strong condition isn't required in the proof.
Here we demonstrate how their argument suffices for the result which is required in this present context.
As described previously, we started with a skew-product \(f_\varphi: X \times \bS^1 \to  X \times \bS^1 \),
\[
  f_\varphi : (x,u) \mapsto (fx, u + \varphi(x))
\]
and then introduced the induced skew-product \(F_\Phi : Y \times \bS^1 \to  Y \times \bS^1 \),
\[
  F_\Phi : (x,u) \mapsto (Fx, u + \Phi(x))
\]
where \(\Phi(x) = \sum_{k=0}^{R(x)-1}\varphi(f^k x)\) and \(F : x \to f^{R(x)}(x)\) is a full branch Markov map.
Since \(F\) is uniformly expanding and \(\Phi\) satisfies the bounded twist property we have good estimates on the associated transfer operators.
This holds precisely for our setting since it only require the properties of the induced system.

Following Young we will introduce a map on the tower which is a model for \(f\) and then, following Gouëzel~\cite[\S3.1]{Gouezel09}, we introduce a skew-product version of the model.
(See Table~\ref{tab:notation} for a comparison between the notation of the reference and that of the present text.)
\begin{table}[tbp]
  \begin{small}
    \begin{tabular}{l | l l}
      \toprule
                         & Gouëzel~\cite[\S3]{Gouezel09}                                     & Present text                                                                          \\
      \midrule
      NUE map            & \(T: X \to X\)                                                    & \(f: X \to X\)                                                                        \\
      Fibre map          & \(\phi : X \to \bR\)                                              & \(\varphi : X \to \bR\)                                                               \\
      Skew-product       & \(\mathcal{T}:X \times \bS^1 \to  X \times \bS^1 \)               & \(f_\varphi: X \times \bS^1 \to  X \times \bS^1 \)                                    \\
      Base of tower      & \(Y\subset X\)                                                    & \(Y\subset X\)                                                                        \\
      Partition of base  & \(\{W_\ell\}\)                                                    & \(\{Y_\ell\}\)                                                                        \\
      Inducing times     & \(r_\ell\)                                                        & \(R_\ell\)                                                                            \\
      Induced map        & \(T_Y : Y \to Y\)                                                 & \(F : Y  \to  Y \)                                                                    \\
      Induced fibre map  & \(\phi_Y : Y \to \bR\)                                            & \(\Phi : Y \to \bR\)                                                                  \\
      Tower              & \(X^{(n)}\)                                                       & \(\widetilde{X}\)                                                                     \\
      Tower map          & \(U^{(n)}:X^{(n)} \to X^{(n)}\)                                   & \(\tilde{f} :\widetilde{X} \to \widetilde{X} \)                                       \\
      Tower skew-product & \(\mathcal{U}^{(n)}:X^{(n)}\times \bS^1 \to X^{(n)}\times \bS^1\) & \(  \tilde{f}_\varphi:   \widetilde{X}\times \bS^1 \to \widetilde{X} \times \bS^1  \) \\
      \bottomrule
      \addlinespace
    \end{tabular}
  \end{small}
  \caption{Comparison of notation}%
  \label{tab:notation}
\end{table}
Let
\[
  \widetilde{X} = \left\{ (x,\ell): x\in Y, 0\leq \ell < R(x)\right\},
\]
together with the tower map
\[
  \tilde{f}: (x,\ell) \mapsto
  \begin{cases}
    (x, \ell + 1) & \text{if \(\ell+1< R(x)\)}   \\
    (Fx, 0)       & \text{if \(\ell+1 = R(x)\)}.
  \end{cases}
\]
We therefore define the tower skew-product \(\tilde{f}_\varphi:   \widetilde{X}\times \bS^1 \to \widetilde{X} \times \bS^1  \) as
\[
  \tilde{f}_\varphi: (x,\ell,u) \mapsto
  \begin{cases}
    (x, \ell + 1,u)     & \text{if \(\ell+1< R(x)\)}   \\
    (Fx, 0, u +\Phi(x)) & \text{if \(\ell+1 = R(x)\)}.
  \end{cases}
\]
We also write the same definition as
\[
  \tilde{f}_\varphi: (x,\ell,u) \mapsto (\tilde{f}(x,\ell),u+\tilde{\varphi}(x,\ell))
\]
where \(\tilde{\varphi}(x,\ell)\) is equal to \(\Phi(x)\) when \(\ell=R(x)\) and equal to \(0\) otherwise.

Using the tower and the above transfer operator estimates we prove the exponential mixing result.
This requires modification of Gouëzel's argument because of our weaker assumptions on the fibre map \(\varphi\).
In particular we prefer to see the action in the fibre only when we arrive at the top of the tower, not incrementally at each step as is done in the reference.
The (possibly) many-to-one map \(\pi : \widetilde{X} \to X\) is defined as \((x,\ell) \mapsto f^\ell x\).
This has the consequence that \(\pi \circ \tilde{f} = f \circ \pi\).
For convenience, here and subsequently, we use the notation \(S_n\varphi = \sum_{\ell=0}^{n-1}\varphi\circ f^\ell\) and, similarly, \(S_n\tilde{\varphi} = \sum_{\ell=0}^{n-1}\varphi \circ \tilde{f}^\ell\).
Abusing notation since no confusion can arise, let \(\pi : \widetilde{X}\times \bS^1 \to X \times \bS^1\) be defined as \((x,\ell,u) \mapsto (f^\ell x,u + S_\ell \varphi (x))\).
This has the consequence that \(\pi \circ \tilde{f}_{\varphi} = f_{\varphi} \circ \pi\).
For each \(\ell\) we define \(\widetilde{X}_\ell\) to be the subset of \(\widetilde{X}\) such that the second coordinate is equal to \(\ell\). In each case this is a copy of \(\{x \in X : R(x) > \ell\}\).
There is a \(f^R\)-invariant  (\(F\)- invariant) measure \(\tilde{\nu}_0\) on \(\widetilde{X}_0\) (which is a copy of \(Y\)). This then extends to \(\tilde{\nu}\) on \(\widetilde{X}\).
We then define the \(f\)-invariant measure \(\nu\) on \(X\) as \(\nu = \pi_* \tilde{\nu}\).
We denote by \(m\) Lebesgue measure on \(\bS^1\).
The \(f_\varphi\)-invariant measure on \(X \times \bS^1\) is given by \(\nu \times m\).
That it is invariant is a simple consequence of \(f_\varphi\) being an isometry in the second coordinate.

By definition of \(\pi\), \(\tilde{\nu}\) and \(\tilde{f}_\varphi\),
\[
  \abs{ \nu\left( g\circ f_\varphi^n \cdot h  \right) - \nu(g)\nu(h) }
  =
  \abs{ \tilde\nu\left( \tilde{g}\circ \tilde{f}_\varphi^n \cdot \tilde{h}  \right) - \tilde\nu(\tilde{g})\tilde\nu(\tilde{h}) }
\]
where \(\tilde{g} = g \circ \pi\) and similarly for \(h\) (observables on the tower).
Let \(\cL\) denote the transfer operator associated to \(\tilde{f}\).
To take advantage of the possibility of a Fourier decomposition in the neutral direction we write
\[
  g(x,\ell,u) = \sum_{k\in \bZ} \hat{g}_k(x,\ell),
  \quad \text{where} \quad
  \hat{g}_k(x,\ell) = \int g(x,\ell,u) e^{-iku} \ du.
\]
Denote by \(J_n\) the Jacobian associated to \(\tilde{f}\).
The twisted transfer operator is equal to, for any \(h : \widetilde{X} \to \bR\),
\[
  \cM_{k}^{n}h(x)
  =
  \sum_{\tilde{f}^{n}y = x} J_{n}(y) h(y) e^{-ik S_n\tilde{\varphi}(y)}.
\]
In order to study correlation one considers the full transfer operator but then only the diagonal terms remain~\cite[(3.4),(3.5)]{Gouezel09} and so,
\[
  \tilde{\nu}\left(  \tilde{g}\circ \tilde{f}_\varphi^n \cdot \tilde{h}   \right)
  = \sum_{k\in \bZ} \tilde{\nu}\left(  \tilde{g}_{-k} \cdot \cM^n_{k}\tilde{h}_k \right)
\]
Following Gouëzel~\cite[\S3]{Gouezel09} (the operators \(R\), \(T\), \(A\), \(B\), \(C\) are identical with identical notation as the reference) we define the following operators which we later use to reconstruct \(\cM_{n,k}\):
\[
  \begin{aligned}
    R_{n,k}h(x)
     & =
    \sum_{\substack{\tilde{f}^ny = x         \\ y\in Y, \tilde{f}y,\ldots,\tilde{f}^{n-1}y \notin Y, \tilde{f}^{n}y \in Y}}
    J_n(y)h(y) e^{-ikS_n\tilde{\varphi}(y)}, \\
    T_{n,k}h(x)
     & =
    \sum_{\substack{\tilde{f}^ny = x         \\ y\in Y, \tilde{f}^{n}y \in Y}}
    J_n(y)h(y) e^{-ikS_n\tilde{\varphi}(y)}, \\
    A_{n,k}h(x)
     & =
    \sum_{\substack{\tilde{f}^ny = x         \\ y\in Y, \tilde{f}y,\ldots,\tilde{f}^{n}y \notin Y}}
    J_n(y)h(y) e^{-ikS_n\tilde{\varphi}(y)},
  \end{aligned}
\]
\[
  \begin{aligned}
    B_{n,k}h(x)
     & =
    \sum_{\substack{\tilde{f}^ny = x         \\ y,\ldots,\tilde{f}^{n-1}y \notin Y, \tilde{f}^{n}y \in Y}}
    J_n(y)h(y) e^{-ikS_n\tilde{\varphi}(y)}, \\
    C_{n,k}h(x)
     & =
    \sum_{\substack{\tilde{f}^ny = x         \\ y,\ldots,\tilde{f}^{n}y \notin Y}}
    J_n(y)h(y) e^{-ikS_n\tilde{\varphi}(y)}.
  \end{aligned}
\]
In words these are, respectively, the cases where:
(\(R\)) The orbit \(y,\tilde{f}y\ldots,\tilde{f}^{n}y\) starts and ends in \(Y\) but isn't in \(Y\) in the meantime;
(\(T\)) The orbit starts and ends in \(Y\);
(\(A\)) The orbit starts in \(Y\) but doesn't finish in \(Y\);
(\(B\)) The orbit is not in \(Y\) until the last iterate when it is in \(Y\);
(\(C\)) The orbit is never in \(Y\).
Consequently, cutting the orbit at the first and last time it belongs to \(Y\) means that
\[
  \cM_{k}^{n}h(x)
  =
  C_{n,k} + \sum_{a+i+b=n} A_{a,k} T_{i,k}B_{b,k}
\]
and cutting according to each time the orbit belongs to \(Y\) implies that
\[
  T_{n,k} = \sum_{p=1}^{\infty} \sum_{j_1+\cdots +j_p=n}
  R_{j_1,k} \cdots R_{j_p,k}.
\]

Since we will work with observables \(h\) which are supported in \(Y\) we can discard the operators \(B_{n,k}\) and \(C_{n,k}\).
The major part of the argument is the study of the operators \(T_{n,k}h\)~\cite[\S3.3]{Gouezel09} and consequently the result of exponential mixing on the tower~\cite[Theorem 3.6]{Gouezel09}.
The same result holds in the present setting because we defined the dynamics on the tower in such a way that we see the action in the fibre only when we arrive at the top of the tower, not incrementally at each step as is done in the reference.
Moreover the assumption on \(\Phi\) match those required in the reference.
Finally we must understand the argument which deduces exponential mixing for \(f_\varphi: X \times \bS^1 \to  X \times \bS^1 \) from exponential mixing of  \(\tilde{f}_\varphi:   \widetilde{X}\times \bS^1 \to \widetilde{X} \times \bS^1  \).
Consider some \(\cC^r\) observable \(g: Y\times \bS^1 \to \bC\) and the corresponding observable on the tower \(g \circ \pi : \widetilde{X}\times \bS^1 \to \bC\).
Note that \(\pi\) is defined differently here compared to in the reference since it must compensate for the fact that the tower only sees the action in the fibre at the top of the tower (i.e., \((x,\ell,u) \mapsto (f^\ell x,u + S_\ell \varphi (x))\)).
However, since we work with observables supported on \(Y\), the base of the tower, we don't see this discrepancy.
Consequently the argument of the reference, with the modifications of the present setting, proves Theorem~\ref{thm:exp-one}.

\section{Twist control}
\label{sec:twist}

In this section we show that the mild control of singular behaviour of the fibre map suffices to give good control for the twist of iterates at hyperbolic times. This type of argument was previously used by Araújo \& Varandas~\cite[\S4.2.2]{AV12} and the later results on Lorenz flows relied on it~\cite{AMV15,AM17}.
In this section we use it in order to prove Theorem~\ref{thm:exp-two}.

Let \(X\) be a compact Riemannian manifold (possibly with boundary), endowed with a Borel measure \(\mu\) (called the reference measure).
For the purposes of this section we suppose that the base transformation \(f:X \to X\) is a nonsingular transformation and that the fibre map \(\varphi: X \to \bS^1\) is piecewise \(\cC^1\) in the sense of being \(\cC^1\) on the open partition elements.
The object of interest is the skew-product \(f_\varphi: X \times \bS^1 \to  X \times \bS^1 \),
\[
  f_\varphi : (x,u) \mapsto (fx, u + \varphi(x)).
\]
It would be convenient to assume that \(\norm{\smash{D(\varphi \circ {f|}_\omega^{-1})}}\) is uniformly bounded (where \(\omega \subset X\) is any open set such that \(f: \omega \to X\) is invertible) because this would imply the existence of a cone field which is forward invariant under \(f_\varphi\) (in the sense that the cones are mapped within themselves) and uniformly bounded away from the neutral direction (see e.g., \cite{BE17}).
Unfortunately, in cases like the one we wish to consider here, it would be impossible to have such a uniform bound.
The reality is that, at a full measure set of points, any tangent vector will approach arbitrarily close to the neutral direction under the action of the partially hyperbolic dynamics.
There is no reason to believe that this hinders good statistical properties but it causes a difficulty we some of the machinery we would like to use.

We consider the singularity points of the fibre map as artificial singularities of the base map and use the notion of hyperbolic times (Definition~\ref{def:hyp-time}).
At hyperbolic times we know that the orbit hasn't been too often too close to any singularity and this is sufficient for our purposes.

\begin{lemma}
  \label{lem:twist-control}
  Suppose that there exists \(C>0\), \(s\geq 0\) such that, for all \(x\),
  \[
    \norm{D\varphi(x)Df(x)^{-1}} \leq C \dist{x}{S}^{-s}.
  \]
  Suppose that \(b\in (0,s^{-1})\), \(\sigma \in (0,1)\), \(\delta>0\).
  There exists \(C'>0\) such that, whenever \(n\) is a \((b,\sigma,\delta)\)-hyperbolic time for \(x\) then, letting \(\Phi(x) = \sum_{k=0}^{n-1}\varphi(f^k x)\),
  \[
    \norm{D\Phi(x)Df^n(x)^{-1}} \leq C'.
  \]
\end{lemma}

\begin{proof}
  We observe that, since \(\Phi(x) = \sum_{\ell=0}^{n-1}\varphi(f^\ell x)\),
  \[
    D\Phi(x)Df^n(x)^{-1} = \sum_{\ell=0}^{n-1} D\varphi(f^\ell x)Df(f^\ell x)^{-1} Df^{n-\ell-1}(f^{\ell+1}x)^{-1}.
  \]
  We can use the assumption of the theorem which controls the quantity \( D\varphi(f^\ell x)Df(f^\ell x)^{-1}\).
  Consequently
  \[
    \norm{D\Phi(x)Df^n(x)^{-1}}\leq C \sum_{\ell=0}^{n-1}  \dist{f^{\ell+1} x}{S}^{-s}  \norm{\smash{Df^{n-\ell-1}(f^{\ell+1}x)^{-1}}}.
  \]
  Observe that \(\norm{\smash{Df^{n-\ell-1}(f^{\ell+1} x)^{-1}}} \leq \prod_{j=\ell+1}^{n-1} \norm{Df(f^jx)^{-1}}\).
  Since, by assumption,  \(n\) is a \((b,\sigma,\delta)\)-hyperbolic time for \(x\), this quantity is bounded from above as \(\prod_{j=\ell+1}^{n-1} \norm{\smash{Df(f^jx)^{-1}}} \leq \sigma^{n-\ell-1}\).
  Additionally \(\ddist{\delta}{f^{\ell+1} x}{S} \geq \sigma^{b(n-\ell-1)}\).
  We observe that \(\dist{\cdot}{\cdot} \geq C \ddist{\delta}{\cdot}{\cdot}\) for some \(C>0\) and immediately absorb this quantity into the previous \(C\).
  Since, by assumption, \(bs < 1\),
  \[
    \begin{aligned}
      \norm{D\Phi(x)Df^n(x)^{-1}}
       & \leq C \sum_{\ell=0}^{n-1}  \sigma^{-b(n-\ell-1)s}  \sigma^{n-\ell-1} \\
       & = C \sum_{k=0}^{n-1}  \sigma^{(1-bs)k}
      \leq \frac{C}{1-\sigma^{1-bs}}.
    \end{aligned}
  \]
  This estimate is independent of \(x\) and \(n\), as required by the statement of the lemma.
\end{proof}

\begin{proof}[Proof of Theorem~\ref{thm:exp-two}]
  The combination of Lemma~\ref{lem:twist-control} and Theorem~\cite[Proposition 1.16]{Gouezel09} is the claimed result.
\end{proof}

\begin{remark}
  As per the following example, \(\norm{D\varphi(x)Df(x)^{-1}}\) might be unbounded even when \(\varphi\) is a bounded function.
  Let \(X = \bR / \bZ\) and let \(f: X \to X\) be defined as \(x \mapsto 2x\).
  The fibre map  \(\varphi : X \to \bR\) is defined for the fundamental domain, \(x\in [0,1)\),
  \[
    \varphi(x) =
    \begin{cases}
      2 - 2x - (1-2x)^{\frac{1}{2}} & \text{if $x<\tfrac{1}{2}$}      \\
      2 - 2x + (2x-1)^{\frac{1}{2}} & \text{if $x\geq \tfrac{1}{2}$}.
    \end{cases}
  \]
  Observe that \(\varphi\) is smooth and continuous.
  However \(\norm{D\varphi(x)Df(x)^{-1}}\) is unbounded at \(x=\frac{1}{2}\).
\end{remark}

\section{Enlarging the singularity set}
\label{sec:artificial}

In this section we describe the relevant argument to use if we have a uniformly expanding transformation and then we need to enlarge the singularity set because of the singularities of the fibre map.
We can then use a large deviations argument to show that the system is NUE in the sense of a controlled singular set (Definition~\ref{def:NUE2}) even with the enlarged singular set.

\begin{lemma}
  \label{lem:singular}
  Suppose that \(f:X\to X\) is \emph{NUE in the sense of a controlled singular set} \(S \subset X\) (Definition~\ref{def:NUE2}).
  Further suppose that \(S' \subset X \setminus S\) is a finite union of \(\cC^1\) manifolds with boundary of dimension strictly less that the dimension of \(X\).
  Let \(S'' = S \cup S' \subset X\).

  Then \(f:X\to X\) is \emph{NUE in the sense of a controlled singular set}, with respect to the set \(S'' \subset X\).
\end{lemma}

\begin{proof}
  Observe that \(S \cap S' = \emptyset\).
  Consequently all the inequalities of assumption \ref{it:S1} remain satisfied.
  Also, the second part of assumption \ref{it:S2} (relating to \(\cQ_{\epsilon,N}\)) remains satisfied since it doesn't depend on the singular set.
  It remains to consider the property which is sometimes described as \emph{slow recurrence to the singular/critical set}.
  We must show that the set
  \[
    \left\{ x \in X :
    \frac{1}{n}\sum_{k=0}^{n-1}- \log
    \ddist{\delta(\epsilon)}{f^k x}{S'} > \epsilon, \
    \text{for some \(n \geq N\)}\right\}
  \]
  is small in Lebesgue measure.
  We can see this estimate as a question of \emph{large deviations} where our ``observable'' is \(x \mapsto - \log \ddist{\delta(\epsilon)}{x}{S'}\).
  Consequently the desired estimate follows from the relevant large deviations result~\cite[Theorem E]{AFLV11}.
\end{proof}

\section{Uniform non-integrability}
\label{sec:uni}

In this section we show how information about the original map can be used to show that the induced fibre map is not a coboundary with respect to the induced base map.
In the terminology used in several of the key references, we show that the \emph{uniform non-integrability} (UNI) condition holds.

As far as this author is aware, there are just two different ways that the fibre map is not cohomologous to a locally constant function.
One approach is to check using periodic orbits and obtain a contradiction (e.g., \cite[Remark 1.15 / Lemma 6.5 / Lemma A.8 / 1st paragraph of \S1.4]{Gouezel09}).
Such arguments are also convenient for establishing that the fibre map not being cohomologous to a locally constant function can be obtained by arbitrary small perturbations (e.g., \cite{AV12,ABV14,BW20}).
And alternative approach is to take advantage of the unbounded nature of the fibre map (or its derivative) in order to obtain a contradiction and hence prove that the fibre map is not cohomologous to a locally constant function.
(See e.g., \cite[\S 4.2.3 \& erratum]{AV12}, \cite[Prop 3.4]{AMV15} and \cite[Lem 4.2, Cor 4.3]{AM16}, often using some type of Livšic type argument~\cite{BHN05}.)
In this section we take this point of view and try to exploit the unbounded nature of the fibre map in order to prove the required property.

Firstly we introduce an example to remind ourselves that we need to take care in this argument.
Let \(X = \bR / \bZ\) and let \(f: X \to X\) be defined as \(x \mapsto 2x\).
Fix some \(a>0\) and define, for \(x\in [0,1]\), the fibre map  \(\varphi(x) = (x)^{-a} - (f x)^{-a}\).
Observe that \(\varphi(\epsilon) \to +\infty\) as \(\epsilon\to 0\) and \(\varphi( \frac{1}{2} + \epsilon) \to -\infty\) as \(\epsilon\to 0\).
Consequently \(\varphi\) is unbounded yet, by definition, is cohomologous to a constant.
The skew-product, although rather disguised, is simply the identity in the fibre.

\begin{remark}
  We can also construct an example which is inspired by the Lorenz flow.
  Let \(f: [-1,1] \to [-1,1]\) be a ``Lorenz-like'' map \cite{LMP05}.
  In particular unbounded derivative at \(x=0\).
  Furthermore, for \(x\in [-1,1]\), let \(\varphi(x) = \log\abs{fx} - \log\abs{x}\).
  The functions goes to \(+\infty\) at \(x=0\), like seen in the Lorenz case, but, differently to the Lorenz case, goes to \(-\infty\) at the two preimages, \(f^{-1}(0)\).
  (The proof of \cite[Theorem 3.4]{AMV15} considers \(f(0^{+})\), \(f^2(0^{+})\) and so would not see the difference with the present example and so it is subtle where such is ruled out in that work where not being cohomologous to a locally constant function is proved. However there they take advantage of the possibility of a Young tower where the base is an open interval containing the singularity \cite[Theorem 4.3]{AV12}.)
\end{remark}

For the remainder of this section we consider the setting assumed in Theorem~\ref{thm:example}.
In particular, \(X = [0,1]\) and \(f:X\to X\) is defined as \(f: x \mapsto 2x \mod 1\).
Furthermore \(\varphi(x) = {\dist{x}{1}}^{a}\) for some \(a\in (0,1)\).
We will take advantage of the fact that in this setting we are able to define an explicit inducing scheme.
Let \(Y=(0,\frac{1}{2})\subset X\), and, for all \(\ell\in\bN\), let \(a_\ell = 2^{-1} - 2^{-\ell}\), \(Y_\ell = (a_{\ell},a_{\ell+1})\).
By definition, \(\{Y_\ell\}_{\ell\in\bN}\) is a partition of a full measure subset of \(Y\).
Moreover \(f^j Y_\ell \cap Y = \emptyset\) whenever \(1\leq j \leq \ell-1\) and  \(f^\ell : Y_\ell \to (0,1)\) is a bijection.
In words, \(\ell\) is the first return time to \(Y\) for each \(x\in Y_\ell\).
Following the notation earlier in this work, \(R(x)\) is defined to be \(\ell\) for each \(x\in Y_\ell\) and \(F : Y \to Y\) is defined as \(F: x\mapsto f^{R(x)} x\).
Let \(S = \{1\} \subset X\).
Although this point is not a singularity in any sense for \(f\) it is a singularity for the fibre map \(\varphi\).

\begin{lemma}
  \label{lem:hyp-times}
  For all \(b>1\) there exists \(\delta>0\) such that, for all \(\ell\in \bN\), \(\ell\) is a \((b,\frac{1}{2},\delta)\)-hyperbolic time for \(x\in Y_\ell\) (with respect to the map \(f : X\to X\) and the singularity set \(S = \{1\} \subset X\)).
\end{lemma}

\begin{proof}
  Let \(x\in Y_\ell\).
  For the first property of hyperbolic times, we observe that \(\prod_{j=\ell-k}^{\ell-1} \norm{Df(f^jx)^{-1}} = 2^{-k}\), consistent with the choice of \(\sigma = \frac{1}{2}\).
  For the other property we must consider how close orbits can approach the singularity set.
  It remains to show that \(\ddist{\delta}{f^{\ell-k}x}{1} \geq \sigma^{bk}\).
  It suffices to consider \(k < \ell\) since \(k=\ell\) implies that \(\dist{f^{\ell-k}x}{1} =\dist{x}{1} \geq \frac{1}{2}\).
  We calculate that,
  \[
    \begin{aligned}
      \dist{f^{\ell-k}x}{1} & \geq \dist{f^{\ell-k}a_{\ell+1}}{1} = 1 - f^{\ell-k}(2^{-1} - 2^{-(\ell+1)}) \\
                            & = 1 - (1 - 2^{-(k+1)}) = \tfrac{1}{2}2^{-k}.
    \end{aligned}
  \]
  Since we assumed that \(b>1\) by choosing \(\delta>0\) we obtain, uniformly, the required estimate (i.e., \(\tfrac{1}{2}2^{-k} \geq 2^{-bk}\) whenever \(k\) is sufficiently large that \(\tfrac{1}{2}2^{-k}  < \delta\)).
\end{proof}

\begin{lemma}
  \label{lem:not-cohomol}
  The induced fibre map \(\Phi : Y \to \bR\) is \emph{not} cohomologous to a function constant on each \(Y_\ell\).
\end{lemma}

\begin{proof}
  Let \(x \in \overline{Y}_1\), \(x'\in \overline{Y}_2\) be defined as the points which satisfy \(F x =x\), \(F x' = x'\).
  Let \(y \in Y_1\) be such that \(y' = F y \in Y_2\) and \(F y' = F^2 y = y\).
  Explicitly,
  \begin{align*}
    x  & = 0, \quad f x = x,                                                                            \\
    x' & = \tfrac{1}{3},  \quad fx' = \tfrac{2}{3}, f^2 x' = x',                                        \\
    y  & = \tfrac{1}{7}, \quad  y' = f y = \tfrac{2}{7}, \quad  f^2 y = \tfrac{4}{7},  \quad f^3 y = y.
  \end{align*}
  Suppose, for the sake of contradiction, that \(\Phi : Y \to \bR\) is cohomologous to a locally constant function in the sense that there exists a \(\cC^1\) function \(\widetilde{\Phi}:Y \to \bR\) such that \(\Phi - \widetilde{\Phi} + \widetilde{\Phi} \circ F\) is constant on each set \(Y_\ell\) (and this extends to \(\overline{Y}_\ell\)).
  Considering the three periodic orbits introduced above, this implies that \( \Phi(x) + \Phi(x') = \Phi(y) + \Phi(y')\) and so,
  \begin{equation}
    \label{eq:fibre-sum}
    \varphi(x) +  \varphi(x')  +  \varphi(f x') = \varphi(y) + \varphi(y')+ \varphi(f y')
  \end{equation}
  Since \(\varphi(x) = {\dist{x}{1}}^{a}\) this implies that
  \[
    \varphi(0) +  \varphi(\tfrac{1}{3})  +  \varphi(\tfrac{2}{3}) - \varphi(\tfrac{1}{7}) - \varphi(\tfrac{2}{7}) - \varphi(\tfrac{4}{7})
    = \chi(a) = 0
  \]
  where, for convenience we defined
  \[
    \chi(a) = 1 + (\tfrac{2}{3})^a + (\tfrac{1}{3})^a - (\tfrac{6}{7})^a - (\tfrac{5}{7})^a - (\tfrac{3}{7})^a.
  \]
  An analysis of this function shows that \(\chi(0)=\chi(1)=0\) but that \(\chi(a)<0\) for all \(a \in (0,1)\).
  This completes the required contradiction \eqref{eq:fibre-sum}.
\end{proof}

\begin{remark}
  Contradicting the equality \eqref{eq:fibre-sum} obtained during the previous proof was the crucial step.
  The same argument could be performed choosing  \(x \in \overline{Y}_n\), \(x'\in \overline{Y}_m\) then the equation to contradict would become,
  \[
    \sum_{j=0}^{n - 1} \varphi(f^j x)- \varphi(f^j y) = \sum_{j=0}^{m - 1} \varphi(f^j y') - \varphi(f^j x').
  \]
  However this, or similar, can be contradicted by many different choices of assumptions on \(\varphi\).
  For example, if \(\varphi\) were constant on \(Y\) but monotone elsewhere and strictly increasing in a neighbourhood of \(1\) the required contradiction would also hold.
\end{remark}

\begin{proof}[Proof of Theorem~\ref{thm:example}]
  Lemma~\ref{lem:hyp-times} implies that, for any \(b>1\) there exists \(\delta>0\) such that the induced system has return times with are \((b,\frac{1}{2},\sigma)\)-hyperbolic times.
  Since \(\varphi(x) = {\dist{x}{1}}^{a}\) we know that \(\norm{D\varphi(x)Df(x)^{-1}} \leq \frac{1}{2} \dist{x}{S}^{-(1-a)}\).
  We may choose \(b\in (0, (1-a)^{-1})\) and so Lemma~\ref{lem:twist-control} applies and proves the required control on \(D\Phi\).
  This, together with the proof that \(\Phi\) is not cohomologous to a locally constant function, as shown in Lemma~\ref{lem:not-cohomol}, means that Theorem~\ref{thm:exp-two} applies in the setting and gives the proof of exponential mixing.
\end{proof}

\section*{Acknowledgements}

\begin{small}
  Massive thanks to Zeze Pacifico, this work would not have existed if it wasn't for them.
  Thanks to Roberto Castorrini, Stefano Galatolo and Carlangelo Liverani for several helpful discussions and comments.
  This work was partially supported by PRIN Grant ``Regular and stochastic behaviour in dynamical systems" (PRIN 2017S35EHN) and the MIUR Excellence Department Project MatMod@TOV awarded to the Department of Mathematics, University of Rome Tor Vergata.
\end{small}


\end{document}